\documentclass[11pt]{article}
\usepackage{geometry}
\usepackage{color}
\usepackage{float}
\usepackage{textcomp}
\usepackage{url}
\usepackage{amsthm}
\usepackage{amsmath}
\usepackage{amssymb}%
\usepackage{mdframed}
\usepackage{multirow}
\usepackage{changes}
\usepackage{hyperref}
\usepackage{authblk} 

\usepackage{mathtools}
\usepackage{booktabs}    
\usepackage{subfigure} 
\usepackage{caption} 
\usepackage{multirow}
\usepackage{multicol}    
\usepackage{array}       
\usepackage{graphicx}
\graphicspath{{figuras/}}
\usepackage{empheq,amsmath}

\usepackage
[ 
lined, 
boxruled, 
commentsnumbered
]
{algorithm2e}
\DecMargin{0.1cm} 

\usepackage{cleveref}

\newtheorem{theorem}{Theorem}[section]
\newtheorem{lemma}[theorem]{Lemma}

\newtheorem{proposition}[theorem]{Proposition}

\newcommand{\R}{\mathbb{R}}

\newcommand{\bi}{\begin{itemize}}
\newcommand{\ei}{\end{itemize}}
\newcommand{\ba}{\begin{array}}
\newcommand{\ea}{\end{array}}

\newlength{\eqAlgoAfter}
\newlength{\eqAlgoBefore}
\begin{document}

\title{\textbf{Enhancing finite-difference-based derivative-free optimization methods with machine learning}}


\author[1]{Timoth\'e Taminiau\thanks{Email: timothe.taminiau@uclouvain.be. Supported Supported by "Fonds spéciaux de Recherche", UCLouvain.}}

\author[1]{Estelle Massart\thanks{Email: estelle.massart@uclouvain.be.}}

\author[1]{Geovani Nunes Grapiglia\thanks{Email: geovani.grapiglia@uclouvain.be. Partially supported by FRS-FNRS, Belgium (Grant CDR J.0081.23).}}

\affil[1]{Université catholique de Louvain, Department of Mathematical Engineering, ICTEAM, 1348 Louvain-la-Neuve, Belgium}

\date{February 10, 2025}

\maketitle

\begin{abstract}
Derivative-Free Optimization (DFO) involves methods that rely solely on evaluations of the objective function. One of the earliest strategies for designing DFO methods is to adapt first-order methods by replacing gradients with finite-difference approximations. The execution of such methods generates a rich dataset about the objective function, including iterate points, function values, approximate gradients, and successful step sizes. In this work, we propose a simple auxiliary procedure to leverage this dataset and enhance the performance of finite-difference-based DFO methods. Specifically, our procedure trains a surrogate model using the available data and applies the gradient method with Armijo line search to the surrogate until it fails to ensure sufficient decrease in the true objective function, in which case we revert to the original algorithm and improve our surrogate based on the new available information. As a proof of concept, we integrate this procedure with the derivative-free method proposed in (Optim. Lett. 18: 195--213, 2024). Numerical results demonstrate significant performance improvements, particularly when the approximate gradients are also used to train the surrogates.
\end{abstract}

\section{Introduction}
\label{section1}

In this work, we address the problem of minimizing a possibly nonconvex function $f:\mathbb{R}^{n}\to\mathbb{R}$ that is bounded from below and is assumed to be differentiable with a Lipschitz continuous gradient. We focus on the black-box setting, where $f(\,\cdot\,)$ is accessible only through function evaluations, i.e., given $x$, one can compute only $f(x)$. The minimization of black-box functions arises, for example, in the calibration of mathematical models that require computer simulations for evaluation, where we want to find the vector $x$ of model parameters that minimizes the misfit $f(x)$ between the model predictions and the available data \cite{chen_development_2020,grimmert_temperature_2023}. Another relevant class of applications is the synthesis of materials or molecules, where $x$ defines the design variables and $f(x)$ measures the discrepancy between the desired and actual material's properties, with the latter accessed through laboratory experiments \cite{frazier_bayesian_2016,korovina_chembo_2020,wang_accelerating_2024}. Optimization problems of this type require Derivative-Free Optimization (DFO) methods \cite{audet_derivative-free_2017,conn_introduction_2009,larson_derivative-free_2019}, which rely solely on function evaluations. 

Classical deterministic DFO methods can be categorized into three main classes: direct-search methods \cite{R3,R4}, model-based trust-region methods \cite{R2,R1}, and finite-difference-based methods \cite{Grapiglia,R5}. At the $k$-th iteration of a direct-search method, the objective function is evaluated at several poll points around the current iterate $x_{k}$. If one of these poll points yields a function value sufficiently smaller than $f(x_{k})$, it is accepted as the new iterate $x_{k+1}$. Otherwise, the method sets $x_{k+1}=x_{k}$ and evaluates a new, closer set of poll points.
In contrast, model-based trust-region methods construct an interpolation model of the objective function using function values at points near $x_{k}$. This model is minimized within a trust region, typically a ball centered at $x_{k}$. If the resulting solution provides sufficient decrease in the objective function, it is accepted as $x_{k+1}$; otherwise, the method sets $x_{k+1}=x_{k}$, reduces the trust-region radius, and updates the interpolation model accordingly.
Finite-difference-based methods, on the other hand, adapt standard first-order optimization techniques (e.g., Gradient Descent) by replacing the gradient $\nabla f(x_{k})$  with a finite-difference approximation $g(x_{k})$, combined with careful choices for the stepsizes. Each iteration of a deterministic DFO method usually requires $\mathcal{O}(n)$ evaluations of the objective function $f(\,\cdot\,)$, which are used to evaluate poll points, construct interpolation models, or compute finite-difference gradient approximations, depending on the type of method. However, very often in DFO problems, function evaluations are computationally expensive because they involve time-consuming computer simulations or laboratory experiments. Therefore, a key objective in DFO is to design methods that can efficiently approximate solutions with as few function evaluations as possible.

The execution of a DFO method generates an increasingly rich dataset about the objective function, comprising iterate points with associated function values, successful trust-region radii or step sizes, and potentially gradient approximations. In recent years, researchers have investigated the possibility of leveraging this data to improve the performance of DFO methods. The central idea is to use the collected data to train surrogate models and then use these models to identify promising candidate points or search directions, at a very low cost in terms of evaluations of $f(\,\cdot\,)$. For example, \cite{irwin_neural_2023} proposed the use of neural network surrogates to accelerate the Implicit Filtering DFO method \cite{kelley_implicit_2011}. Specifically, at the $k$-th iteration of their new method, NNAIF (Neural Network Accelerated Implicit Filtering), the dataset $\mathcal{F}=\left\{\left(y_{i},f(y_{i})\right)\right\}_{i=1}^{N}$ of all previously computed points with associated function values is used to train a neural network surrogate $m_{\theta}(\,\cdot\,)$ by approximately solving the problem
\begin{equation}
\min_{\theta}\,\,L_{\cal F}(\theta)\equiv \dfrac{1}{N}\sum_{i=1}^{N}\left(m_{\theta}(y_{i})-f(y_{i})\right)^{2}.
\label{eq:gg1}
\end{equation}
Then, Gradient Descent is applied to minimize $m_{\theta}(\,\cdot\,)$ with a fixed budget of iterations, returning a candidate point $x_{k}^{+}$. If 
\begin{equation*}
 f(x_{k}^{+})<f(x_{k})-\epsilon
\end{equation*}
for a user defined $\epsilon>0$, the method sets $x_{k+1}=x_{k}^{+}$. Otherwise, $x_{k+1}$ is computed from $x_{k}$ via a standard iteration of the Implicit Filtering method. The authors established a liminf-type global convergence result for NNAIF and reported numerical results on 31 noisy problems from the CUTEst collection, where NNAIF achieved significant improvements over the standard Implicit Filtering method on large-scale problems. In a different direction, \cite{giovannelli_limitation_2023} investigated surrogate-based variants of a Full-Low-Evaluation DFO method \cite{berahas_full-low_2023}, which originally relies on BFGS search directions computed using finite-difference gradient approximations. Specifically, instead of computing a finite-difference approximation of $\nabla f(x_{k})$, they used a dataset of the form $\mathcal{F}$ described above to train a surrogate model $m_{\theta}(\,\cdot\,)$ by solving (\ref{eq:gg1}). The search direction was then computed as 
\begin{equation*}
    p_{k}=-H_{k}\nabla m_{\theta}(x_{k}),
\end{equation*}
where matrix $H_{k}$ was obtained from $H_{k-1}$, $s_{k-1}=x_{k}-x_{k-1}$ and $y_{k-1}=\nabla m_{\theta}(x_{k})-g(x_{k-1})$, with $g(x_{k-1})$ being the gradient approximation used in the previous iteration. The authors conducted numerical experiments on a subset of CUTEst problems to test various types of surrogates, including polynomials, radial basis functions, and neural networks with different activation functions. However, none of the surrogate variants was able to outperform the base method that relies on finite-difference gradient approximations. 

Inspired by the promising results of \cite{irwin_neural_2023}, in this work we propose a new heuristic to enhance the performance of virtually any finite-difference-based DFO methods using surrogate models. Given a base method, our heuristic is invoked at the end of each (outer) iteration. Instead of relying solely on $\mathcal{F}$, it employs the augmented dataset $\mathcal{F}\cup\mathcal{G}$, with $\mathcal{G} = \left\{(z_{j},g(z_{j}))\right\}_{j=1}^{M}$, where $g(z_{j})\approx \nabla f(z_{j})$, to train a model $m_{\theta}(\,\cdot\,)$ by (approximately) solving the problem
\begin{align}
\min_{\theta}\,\,L_{\mathcal{F},\mathcal{G}}(\theta)\equiv \dfrac{1}{N}\sum_{i=1}^{N}\left(m_{\theta}(y_{i})-f(y_{i})\right)^{2}+\dfrac{1}{M}\sum_{j=1}^{M}\|\nabla m_{\theta}(z_{j})-g(z_{j})\|_{2}^{2} +\lambda\|\theta\|^{2}_{2},
\label{eq:gg2}
\end{align}
a technique commonly referred to as Sobolev learning in the Machine Learning community \cite{czarnecki_sobolev_2017}. Once the surrogate is trained, our heuristic applies Gradient Descent with Armijo line search to $m_{\theta}(\,\cdot\,)$ until it fails to ensure a sufficient decrease of the true objective function, returning then the best point found, which is taken as the next iterate point. As a proof of concept, we integrate this procedure with the derivative-free method proposed by \cite{grapiglia_worst-case_2023}. We prove that the enhanced method requires no more than $\mathcal{O}\left(n\epsilon^{-2}\right)$ function evaluations to find an $\epsilon$-approximate stationary point of $f(\,\cdot\,)$, where the constant factor in the complexity bound is inversely proportional to the average number of successful surrogate steps per outer iteration. Numerical experiments using radial basis functions and neural network surrogate models demonstrate significant performance improvements compared to the base method.

The paper is organized as follows. In \Cref{sec:surrogate_framework}, we present our surrogate-based heuristic and show how it can be incorporated into a DFO method based on finite differences. We establish a worst-case complexity bound that reflects the performance improvement that successful surrogate steps may produce. In \Cref{sec:models} we describe two classes of models that can be used as surrogates, namely, neural networks and radial basis functions. We also propose there a new interpretation of Sobolev learning (when combined with finite-difference gradients) as a form of model regularization that penalizes the curvature of the model. Finally, in \Cref{sec:numerics}, we present numerical results illustrating the performance gains achieved with our proposed heuristic.

\section{Surrogate-based heuristic}
\label{sec:surrogate_framework}

This section presents our surrogate-based heuristic, described in \Cref{algorithm3}, that can be integrated to any derivative-free optimization algorithm that relies on the (generally costly) computation of an approximation $g(x) \approx \nabla f(x)$.

\begin{algorithm}[H]
   \caption{\texttt{Surrogate}($v$, $f$, $\mathcal{F}$, $\mathcal{G}$, $\sigma$, $\rho$, $\lambda$, $\gamma$, $\epsilon$)}
   \label{algorithm3}
   \noindent\textbf{Inputs:} A reference point $v \in \mathbb{R}^n$, a zeroth-order oracle for $f(\,\cdot\,)$, datasets $\mathcal{F} = \{(y_i, f(y_i))\}_{i=1}^N$ and $\mathcal{G} = \{(z_j, g(z_j))\}_{j=1}^M$ with $(v, f(v)) \in \mathcal{F}$, and constants $\sigma, \rho, \lambda, \gamma, \epsilon > 0$.
   \\[0.15cm]
   \noindent\textbf{Step 1:}  Train a continuously differentiable surrogate model $m_\theta : \mathbb{R}^n \rightarrow \mathbb{R}$ by (approximately) solving (\ref{eq:gg2}). Set $v_0 = v$, $L_0 = \sigma$, $\mathcal{F}^+:= \emptyset$ and $t:= 0$.
   \\[0.15cm]
   \noindent\textbf{Step 2:} Find the smallest integer $\ell_t \geq 0$ such that the point 
   \begin{align*}
    \hat{v}_{t}=v_{t}-\frac{1}{2^{\ell_{t}}L_{t}}\nabla m_{\theta}(v_{t})
   \end{align*}
   satisfies the inequality
   \begin{align*}
   m_{\theta}(v_{t})-m_{\theta}(\hat{v}_{t})\geq\dfrac{\rho}{2^{\ell_{t}}L_{t}}\|\nabla m_{\theta}(v_{t})\|_{2}^{2}.    
   \end{align*}

\noindent\textbf{Step 3:} Compute $f(\hat{v}_{t})$ and set $\mathcal{F}^+ := \mathcal{F}^+ \cup \{(\hat{v}_t, f(\hat{v}_{t}))\}$.
\\[0.15cm]
\noindent\textbf{Step 4:} If
\begin{equation} \label{eq:decr_per_step}
    f(v_t) - f(\hat{v}_{t}) \geq \frac{1}{\gamma \sigma} \epsilon^2, 
\end{equation}
define $v_{t+1} = \hat{v}_t$, $L_{t+1} = 2^{\ell_t-1} L_t$, set $t := t + 1$ and go back to Step 2. Otherwise, set $t^+=t$, $v^+=v_t$ and STOP.
\\[0.15cm]
\noindent\textbf{Return} $v^+,t^+,\mathcal{F}^+$.
\end{algorithm}

Notice that \eqref{eq:decr_per_step} ensures that
\begin{equation*}
f(v) - f(v^{+}) \geq \frac{t^{+}}{\gamma \sigma} \epsilon^2.
\end{equation*}
Therefore, from now on we will refer to $t^+$ as the number of \textit{successful surrogate steps}. In order to prove finite termination of \Cref{algorithm3}, we will consider the following assumption. \\

\noindent \textbf{Assumption A1} $f(\,\cdot\,)$ is bounded from below by $f_\text{low} \in \R$.

\begin{lemma}
Suppose that A1 holds. Then \Cref{algorithm3} has finite termination. 
\end{lemma}

\begin{proof}
Suppose by contradiction that \Cref{algorithm3} never stops and consider an iteration index $t_{*}$ such that
\begin{equation*}
t_{*}> \max\left\{1,\gamma \sigma (f(v)-f_\text{low}) \epsilon^{-2}\right\}.
\end{equation*}
Then
\begin{equation*}
f(v_t) - f(v_{t+1}) \geq \frac{1}{\gamma \sigma} \epsilon^2 \; \text{for} \; t=0,1, \dots, t_*-1.
\end{equation*}
Summing up these inequalities, it follows that
\begin{equation*}
f(v) - f(v_{t_*}) \geq \frac{t_{*}}{\gamma \sigma}\epsilon^{2}> f(v)-f_\text{low}.
\end{equation*}
This implies that $f(v_{t_{*}}) < f_\text{low}$, contradicting A1.
\end{proof}

To illustrate the applicability of \Cref{algorithm3}, we will integrate it with a simplified version\footnote{The original version of the method includes a Hessian approximation $B_k$ which can be updated with a BFGS formula based on the finite-difference gradient approximation. \Cref{algorithm4} corresponds to the simple case with $B_k=0$ for all $k$.} 
of the DFO method proposed by \cite{grapiglia_worst-case_2023}, which will be referred as the base method. This method is a derivative-free variant of the Gradient Descent method with Armijo line search where $\nabla f(x_k)$ is approximated using forward finite differences. When $f$ is bounded from below and has Lipschitz continuous gradient, this method requires no more than $\mathcal{O}\left(n\epsilon^{-2}\right)$ function evaluations to find an $\epsilon$-approximate stationary point. The detailed steps are described below.
\begin{algorithm}[h!]
\caption{Base Method}
\label{algorithm4}
\textbf{Step 0.} Given $x_0 \in \R^n$, $\sigma_0 \geq \sigma_{\text{min}} > 0$, $\epsilon > 0$, set $k:=0$. 
\\[0.15cm]
\textbf{Step 1.} Set $i:=0$. \\[0.15cm]
\textbf{Step 1.1.} For 
\begin{equation*}
h_i = \frac{2\epsilon}{5 \sqrt{n} (2^i \sigma_k)}
\end{equation*}
compute the gradient approximation $g_{h_i}(x_k)$ defined by 
\begin{equation*}
[g_{h_i}(x_k)]_j = \frac{f(x_k+h_i e_j) - f(x_k)}{h_i}, \; j=1,\dots,n.
\end{equation*}

\textbf{Step 1.2.} If 
\begin{equation*}
\|g_{h_i}(x_k)\|_2 < \frac{4 \epsilon}{5},
\end{equation*}
set $i:=i+1$ and go back to Step 1.1. \\[0.2cm]
\textbf{Step 1.3.} If 
\begin{align*}
f(x_k) - f\left(x_k - \frac{1}{2^i \sigma_k} g_{h_i}(x_k)\right) \geq \frac{1}{8(2^i \sigma_k)} \|g_{h_i}(x_k)\|_2^2
\end{align*}
then set $i_k=i$ and 
\begin{equation*}
x_k^+ = x_k - \frac{1}{2^{i_k}\sigma_{k}} g_{h_{i_k}}(x_k)
\end{equation*}
Otherwise, set $i:=i+1$ and go back to step 1.1. \\[0.15cm]
\textbf{Step 2.} Set $x_{k+1} = x_k^+$, $\sigma_{k+1} = \max \{ 2^{i_k-1} \sigma_k, \sigma_{\text{min}} \}$, $k:=k+1$ and go back to Step 1.
\end{algorithm}

The surrogate-accelerated variant of this algorithm is presented in \Cref{algorithm5}. As it can be seen, incorporating \Cref{algorithm3} in \Cref{algorithm4} requires only small changes in the latter.

\begin{algorithm}[H]
\caption{Base Method with Surrogate Heuristic}
\label{algorithm5}
\textbf{Step 0.} Given $x_0 \in \R^n$, $\sigma_0 \geq \sigma_{\text{min}} > 0$, $\epsilon > 0$, and \fbox{$\rho, \lambda, \gamma > 0$}, set \small\fbox{$\mathcal{F} = \{(x_0,f(x_0))\},\,\, \mathcal{G} = \emptyset$}\normalsize, and $k:=0$.
\\[0.2cm]
\textbf{Step 1.} Set $i:=0$. 
\\[0.15cm]
\textbf{Step 1.1.} For
\begin{equation*}
h_i = \frac{2\epsilon}{5 \sqrt{n} (2^i \sigma_k)}
\end{equation*}
compute the gradient approximation $g_{h_i}(x_k)$ defined by 
\begin{equation*}
[g_{h_i}(x_k)]_j = \frac{f(x_k+h_i e_j) - f(x_k)}{h_i}, \; j=1,\dots,n. \end{equation*}
Set \fbox{$\mathcal{F} := \mathcal{F} \cup \{(x_k+h_i e_j,f(x_k+h_i e_j))\}_{j=1}^n$} 
\\[0.15cm]
\textbf{Step 1.2.} If 
\begin{equation*}
\|g_{h_i}(x_k)\|_2 < \frac{4 \epsilon}{5},
\end{equation*}
set $i:=i+1$ and go back to Step 1.1. 
\\[0.15cm]
\textbf{Step 1.3.} If 
\small
\begin{equation*}
f(x_k) - f\left(x_k - \frac{1}{2^i \sigma_k} g_{h_i}(x_k)\right) \geq \frac{1}{8(2^i \sigma_k)} \|g_{h_i}(x_k)\|_2^2 
\end{equation*}
\normalsize
then set $i_k=i$,
\begin{equation*}
x_k^+ = x_k - \frac{1}{2^{i_k}\sigma_{k}} g_{h_{i_k}}(x_k)
\end{equation*}
and update the datasets:
\begin{equation*}
\fbox{$\mathcal{F} := \mathcal{F} \cup \{(x_k^+, f(x_k^+))\},\quad \mathcal{G} := \mathcal{G} \cup \{(x_k,g_{h_{i_k}}(x_k))\}$.}
\end{equation*}
Otherwise, set $i=i+1$ and go back to step 1.1. 
\\[0.15cm]
\textbf{Step 2.} Attempt performing surrogate steps\\[0.2cm]\small
\fbox{$(v_k^+,t_k,\mathcal{F}^+) =\texttt{ Surrogate}\left(x_k^+, f, \mathcal{F}, \mathcal{G}, 2^{i_k} \sigma_k, \rho, \lambda, \gamma, \epsilon\right)$}
\normalsize
\\[0.3cm]
\textbf{Step 3.} Set \fbox{$x_{k+1} = v_k^+$, $\mathcal{F}:=\mathcal{F}\cup\mathcal{F}^{+}$}, 
$$\sigma_{k+1} = \max \{ 2^{i_k-1} \sigma_k, \sigma_{\text{min}} \},$$ 
$k:=k+1$, and go back to Step 1.
\end{algorithm}

To analyse the worst-case oracle complexity of \Cref{algorithm5}, we will consider the following additional assumption. \\

\noindent \textbf{Assumption A2} $\nabla f(\,\cdot\,)$ is $L$-Lipschitz continuous, that is
\begin{equation*}
\|\nabla f(x)-\nabla f(y)\|_2 \leq L \|x-y\|_2, \; \forall x,y \in \R^n.
\end{equation*}

\begin{lemma} \label{lem:2}
Suppose that A2 holds. Then the $k$-th iteration of \Cref{algorithm5} is well defined whenever $\|\nabla f(x_k)\|_2 > \epsilon$. In addition, if 
\begin{equation*}
\|\nabla f(x_k)\|_2 > \epsilon, \; \text{for} \; k = 0,1, \dots, T-1,
\end{equation*}
for $T \geq 1$, then
\begin{equation} \label{eq:1}
\sigma_k \leq 2 \max \{\sigma_0, 2L\}=\sigma_{\max}, \; \text{for} \; k=0,1, \dots, T
\end{equation}
and
\begin{equation} \label{eq:3}
f(x_k) - f(x_k^+) \geq \frac{1}{2 C_f} \|\nabla f(x_k)\|_2^2, \; \text{for} \; k=0,1, \dots, T-1,
\end{equation}
where
\begin{equation*}
C_f=\frac{81}{8} \max \Big{\{}\sigma_{\max}, \frac{L^2}{\sigma_{\min}} \Big{\}}.
\end{equation*}
\end{lemma}

\begin{proof}
It follows directly from Lemmas 1-4 in \cite{grapiglia_worst-case_2023}.
\end{proof}

Let
\begin{equation*}
T(\epsilon) = \inf \{ k \in \mathbb{N} \; | \; \|\nabla f(x_k)\|_2 \leq \epsilon\}
\end{equation*}
and suppose that $T(\epsilon) \geq 1$. Given $T$, with $0 < T \leq T(\epsilon)$, we will denote by $S(T)$ the average number of successful surrogate steps over the first $T$ iterations of \Cref{algorithm5}, i.e.,
\begin{equation*}
S(T) = \frac{1}{T} \sum_{k=0}^{T-1} t_k.
\end{equation*}

\begin{lemma} \label{lem:3}
Suppose that A1-A2 hold and let $\{x_k\}_{k=0}^{T(\epsilon)}$ be generated by \Cref{algorithm5}. Then
\begin{equation} \label{eq:4}
T(\epsilon) \leq \frac{2 C_{\max} (f(x_0)-f_{\text{low}})}{1 + S(T(\epsilon))} \epsilon^{-2},
\end{equation}
where
\begin{equation*}
C_{\max} = \max \{ C_f, \gamma \sigma_{\max} \}
\end{equation*}
with $C_f$ and $\sigma_{\max}$ defined in \Cref{lem:2}.
\end{lemma}

\begin{proof}
Given $k \in \{ 0, 1, \dots, T(\epsilon)-1 \}$, it follows from \Cref{algorithm3} and inequality \eqref{eq:1} in \Cref{lem:2} that
\begin{equation} \label{eq:2}
f(x_k^+) - f(v_k^+) \geq \frac{t_k}{\gamma (2^{i_k} \sigma_k)} \epsilon^2 \geq \frac{t_k}{2 \gamma \sigma_{\max}} \epsilon^2.
\end{equation}
Thus, combining \eqref{eq:2} with inequality \eqref{eq:3} in \Cref{lem:2}, we obtain
\begin{align*}
f(x_k) - f(x_{k+1}) &= f(x_k) - f(v_k^+) \\
&= f(x_k) - f(x_k^+) + f(x_k^+) - f(v_k^+) \\
&\geq \frac{1}{2 C_f} \|\nabla f(x_k)\|_2^2 + \frac{t_k}{2 \gamma \sigma_{\max}} \epsilon^2 \\
&\geq (1+t_k) \frac{\epsilon^2}{2 C_{\max}}.
\end{align*}
Summing up these inequalities for $k=0,1, \dots, T(\epsilon)-1$ and using A1, it follows that
\begin{align*}
f(x_0) - f_{\text{low}} &\geq f(x_0) - f(x_{T(\epsilon)}) \\
&= \sum_{k=0}^{T(\epsilon)-1} (f(x_k) - f(x_{k+1})) \\
&\geq \Big{(}T(\epsilon) + \sum_{k=0}^{T(\epsilon)-1} t_k\Big{)} \frac{\epsilon^2}{2 C_{\max}} \\
&= (1 + S(T(\epsilon))) T(\epsilon) \frac{\epsilon^2}{2 C_{\max}},
\end{align*}
which implies that \eqref{eq:4} is true.
\end{proof}
Since $S(T(\epsilon))\geq 0$, it follows from \Cref{lem:3} that \Cref{algorithm5} takes at most $\mathcal{O}\left(\epsilon^{-2}\right)$ outer iterations to find an $\epsilon$-approximate stationary point of $f$. The next results gives us a worst-case complexity bound for the number of function evaluations required by \Cref{algorithm5} to achieve this goal.

\begin{theorem} \label{thm:1}
Suppose that A1-A2 hold. Let $\{x_k\}_{k=0}^{T(\epsilon)}$ be generated by \Cref{algorithm5} and denote by $\text{FE}(\epsilon)$ the corresponding number of evaluations of $f(\,\cdot\,)$ that were required. Then
\begin{align}
\text{FE}(\epsilon) &\leq 4 \eta(S(T(\epsilon)) (n+1) C_{\max} (f(x_0) - f_\text{low}) \epsilon^{-2} + \log_2\Big{(}\frac{\sigma_{\max}}{\sigma_0}\Big{)}(n+1) + T(\epsilon), \label{eq:5}
\end{align}
where
\begin{equation*}
\eta(S(T(\epsilon)) = \frac{1 + \frac{S(T(\epsilon))}{2(n+1)}}{1 + S(T(\epsilon))}
\end{equation*}
and $\sigma_{\max}$, $C_{\max}$ are defined in \Cref{lem:2} and \Cref{lem:3}.
\end{theorem}

\begin{proof}
Notice that
\begin{equation}
\text{FE}(\epsilon) \leq \sum_{k=0}^{T(\epsilon)-1} (i_k+1) (n+1) + (t_k+1).
\label{eq:tmg1}
\end{equation}
By the update rule for $\sigma_k$, we have
\begin{equation*}
2^{i_k-1} \sigma_k \leq \max \{ 2^{i_k-1} \sigma_k, \sigma_{\min} \} = \sigma_{k+1}
\end{equation*}
and so
\begin{equation}
i_k + 1 \leq  2 + \log_2 (\sigma_{k+1}) - \log_2 (\sigma_k). \label{eq:tmg2}
\end{equation}
Now combining (\ref{eq:tmg1}) and (\ref{eq:tmg2}), and using \Cref{lem:2} and \Cref{lem:3}, it follows that
\small
\begin{align*}
&\text{FE}(\epsilon) \leq \sum_{k=0}^{T(\epsilon)-1} \left[2 + \log_2\left(\frac{\sigma_{k+1}}{\sigma_k}\right)\right] (n+1) + \sum_{k=0}^{T(\epsilon)-1} (t_k + 1) \\
&= (n+1) \left[2 T(\epsilon) + \log\Big{(}\frac{\sigma_{T(\epsilon)}}{\sigma_0}\Big{)}\right] + T(\epsilon) S(T(\epsilon)) + T(\epsilon) \\
&\leq  (n+1) \left[2\Big{(}1+\frac{S(T(\epsilon))}{2(n+1)}\Big{)} T(\epsilon) + \log_2\Big{(}\frac{\sigma_{\max}}{\sigma_0}\Big{)}\right] + T(\epsilon) \\
&\leq 4 \eta (S(T(\epsilon))) (n+1) C_{\max} (f(x_0) - f_\text{low}) \epsilon^{-2} + \log_2\Big{(}\frac{\sigma_{\max}}{\sigma_0}\Big{)} (n+1) + T(\epsilon).
\end{align*}
\end{proof}
\normalsize
Notice that, $\eta:\mathbb{R}_{+}\to\mathbb{R}$ defined by
\begin{equation*}
\eta(S) = \frac{1 + \frac{S}{2(n+1)}}{1 + S}
\end{equation*}
is a nonnegative decreasing function with $\eta(0) = 1$. Therefore, even when our surrogate heuristic do not produce any successful step (i.e., $S(T(\epsilon))=0$), by \Cref{thm:1}, we have $\text{FE}(\epsilon) = O(n \epsilon^{-2})$, which matches the worst-case complexity bound already established for \Cref{algorithm3}. On the other hand, when $S(T(\epsilon)) > 0$, we have $\eta(S(T(\epsilon))) < 1$, which gives an improvement in the oracle complexity bound \eqref{eq:5}. For this reason, in what follows we will refer to $\eta(S(T(\epsilon)))$ as the \textit{surrogate gain}. In particular, if \fbox{$S(T(\epsilon)) \geq n$}, we have
\begin{equation*}
\eta(S(T(\epsilon))) \leq \dfrac{3}{2(n+1)},
\end{equation*}
which, combined with \eqref{eq:5}, yields
\begin{equation*}
\text{FE}(\epsilon) \leq 6 C_{\max} (f(x_0)-f_\text{low}) \epsilon^2 + \log_2 \Big{(}\frac{\sigma_{\max}}{\sigma_0}\Big{)} (n+1) + T(\epsilon).
\end{equation*}
Therefore, in this ideal case, the dependence of the complexity bound on the problem dimension would be just due to the term $\log_2(\frac{\sigma_{\max}}{\sigma_0}) (n+1)$, which comes from the lack of knowledge of the Lipschitz constant $L$.

\section{Surrogate models of the objective}
\label{sec:models}

The surrogate model $m_{\theta}$ used in \Cref{algorithm3} and \Cref{algorithm5} is trained by solving \eqref{eq:gg2}; this form of model training, that penalizes errors on the model derivatives, is known in the machine learning community as Sobolev learning \cite{czarnecki_sobolev_2017}. We show next that, under some specific assumptions, Sobolev learning can be interpreted as a regularization on the curvature of the surrogate model.

Note that the approximate gradients $g_{h_1}(z_1), \dots, g_{h_M}(z_M)$ were obtained in \Cref{algorithm5} using finite differences:  
\begin{equation} \label{eq:fin_diff}
    [g_{h}(z)]_l = \frac{f(z+h e_l) - f(z)}{h}, \quad \text{for} \; l = 1, \dots, n.
\end{equation} 
This expression requires evaluating $f$ at the $n+1$ points $z_j, z_j+h_j e_1, \dots, z_j+h_j e_n$, so that
\begin{align*}
\{(z_j,f(z_j)), (z_j+h_j e_1,f(z_j+h_j e_1)),\dots, (z_j+h_j e_n,f(z_j+h_j e_n))\}_{ j = 1, \dots, M} \subset \mathcal{F}.
\end{align*}
The following proposition shows that, assuming exact interpolation of the model for all these points, Sobolev learning with forward finite-difference gradients induces a regularization on the curvature of the model.

\begin{proposition} \label{prop1}
Let $m_{\theta} : \R^n \to \R$ be twice continuously differentiable, and $z \in \R^n$ be such that
\begin{align}
m_{\theta}(z) &= f(z) \label{eq:interp1} \\
m_{\theta}(z+h e_l) &= f(z+h e_l), \quad \text{for} \; l = 1, \dots, n. \label{eq:interp2}
\end{align}
Then, there exists $\zeta_1, \dots, \zeta_n \in [0,1]$ such that
\begin{equation*}
\|g_h(z) - \nabla m_{\theta}(z)\|_2^2 = \frac{h^2}{4} \sum_{l=1}^n [\nabla^2 m_{\theta}(z+h \zeta_l e_l)]_{ll}^2,
\end{equation*}
where $g_h(z)$ is defined by \eqref{eq:fin_diff}.
\end{proposition}

\begin{proof}
By the interpolation conditions \eqref{eq:interp1} and \eqref{eq:interp2}, for all $l=1,\dots,n$,
\[[g_h(z)]_l = \frac{f(z+h e_l) - f(z)}{h} = \frac{m_{\theta}(z+h e_l) - m_{\theta}(z)}{h}.\]
By the second-order Taylor expansion of $m_\theta(\,\cdot\,)$ with remainder, there exists $\zeta_l\in [0,1]$ such that
\[m_{\theta}(z+h e_l) = m_{\theta}(z) + h [\nabla m_{\theta}(z)]_l + \frac{h^2}{2} [\nabla^2 m_{\theta}(z+h\zeta_l e_l)]_{ll}, \]
so that 
\[\frac{m_{\theta}(z+he_l)-m_{\theta}(z)}{h} = [\nabla m_{\theta}(z)]_l + \frac{h}{2} [\nabla^2 m_{\theta}(z+h\zeta_l e_l)]_{ll}. \]
There follows that
\[[g_h(z) - \nabla m_{\theta}(z)]_l = \frac{h}{2} [\nabla^2 m_{\theta}(z+h\zeta_l e_l)]_{ll}.\]
\end{proof}

This result implies that, when finite difference gradients are used and all points in $\mathcal{F}$ are exactly interpolated, Sobolev learning can be interpreted as a penalty on a term related to the curvature of the model. 

In this work, we consider two families of models: shallow neural networks (NNs) and radial basis functions (RBFs).

\paragraph*{Shallow Neural Networks.} 
The model $m_{\theta}$ is here chosen as a one-hidden-layer neural network: 
\begin{equation*} \label{eq:nnModel}
m_\theta(x) =  W_2 \phi (W_1 x + b_1) + b_2, 
\end{equation*}
where $\phi : \R^q \rightarrow \R^q$ is an (entrywise) activation function, $W_1 \in \R^{q \times n}, W_2 \in \R^{1 \times q}$ are weight matrices and $b_1 \in \R^q, b_2 \in \R$ are biases. While neural networks are ubiquitous in machine learning, solving \eqref{eq:gg2} for the parameters $\theta = \{W_1, W_2, b_1, b_2\}$ requires in general iterative optimization algorithms, sometimes incurring substantial additional costs.

\paragraph*{Radial Basis Functions.} The model $m_{\theta}$ is chosen as 
\begin{equation} \label{eq:rbf}
m_\theta(x) = \sum_{i=1}^N \alpha_i \psi(\|x-y_i\|_2) + \beta^T x + \delta,
\end{equation}
where $\psi : \R \rightarrow \R$ is a radial basis function,  $\alpha \in \R^N$, $\beta \in \R^n$ and $\delta \in \R$ are parameters, and the point $y_i$ for $i = 1, \dots, N$, are the abscissas of the points in $\mathcal{F}$. Taking $\lambda = 0$, \eqref{eq:gg2} can be efficiently solved for the parameters $\theta = \{ \alpha_1, \dots, \alpha_N, \beta, \delta \}$ by a least-squares solver. 

\section{Numerical results}
\label{sec:numerics}

In this section, we compare our proposed surrogate-acceleration method (\Cref{algorithm5}) to the base method (\Cref{algorithm4}) on the benchmark problem set proposed by \cite{gratton_opm_2021} which is a subset of the well-known CUTEst collection \cite{gould_cutest_2015}.

\subsection{Benchmarking procedure}

The benchmark test set contains 134 problems whose dimension $n$ ranges from 1 to 110 with an average value around 15. To compare the different methods considered, we rely on data profiles with the convergence test
\begin{equation} \label{D}
f(x_0) - f(x) \geq (1-\tau) (f(x_0) - f_{\text{best}}),
\end{equation}
where $x_0 \in \R^n$ is the starting point of the method, $\tau > 0$ (here, $\tau = 10^{-4}$) is a tolerance parameter and $f_{\text{best}}$ is the best function value found among all the methods within a fixed budget of function evaluations; see \cite{more_benchmarking_2009} for more information. Each curve of a data profile gives thus the proportion of problems for which a method found $x \in \R^n$ satisfying the convergence test \eqref{D} depending on the allowed number of function evaluations. In order to account for the difference in dimension among the problems, the number of function evaluations is converted into the number of simplex gradients, where one simplex gradient corresponds to $n+1$ function evaluations.

\subsection{Implementation detail}

\Cref{algorithm4} and \Cref{algorithm5} are implemented with the parameters $\sigma_0 = 1$, $\sigma_{\min} = 10^{-2}$, $\rho=10^{-4}$, $\gamma = \frac{25}{2}$ and $\epsilon = 10^{-5}$. Since training a model based on a large dataset may be computationally expensive or lead to badly-conditioned problems, we limit the number of points in $\mathcal{F}$ and $\mathcal{G}$, ensuring that $N \leq 10 (n+1)$ and $M \leq 10$; when one of these thresholds is reached, the oldest points are removed first.

The shallow NN model is made of one fully connected hidden layer with width $5n$. For the activation functions, following \cite{giovannelli_limitation_2023}, we focused on SoftPlus, SiLU and the sigmoid. Other classical activations such as ReLU were not considered as they do not lead to a differentiable model (see Step 1 of \Cref{algorithm3}). The regularization parameter in the training problem \eqref{eq:gg2} is $\lambda = 10^{-4}$, and \eqref{eq:gg2} is solved with low-memory BFGS, that we stop at iteration $K$ if $\|\nabla L_{\mathcal{F},\mathcal{G}}(\theta_K)\|_2 \leq 10^{-6} \max \{1, \|\nabla L_{\mathcal{F},\mathcal{G}}(\theta_0)\|_2\}$, or when the maximum budget of 1000 iterations is reached. During the first call to \Cref{algorithm3}, the surrogate model is initialized using He initialization \cite{he_delving_2015} for the SoftPlus and SiLU activations, and using Glorot initialization \cite{glorot_understanding_2010} for the sigmoid. All subsequent calls to \Cref{algorithm3} rely on warm-start initialization, i.e., the new surrogate is initialized based on the last trained surrogate.

For the RBF models, we set $\lambda = 0$ in \eqref{eq:gg2} and solved this problem with a direct least-square solver. If there is more than one optimal solution, the one with the coefficients that have the smallest norm is used.


\subsection{Experiments}

Our first experiment highlights the benefits of NN-based surrogate acceleration compared to the base method. For the NN surrogate, we rely on  the SoftPlus activation:
\begin{equation*}
[\phi(x)]_i = \log\big{(} 1 + e^{[x]_i}\big{)}, \; \text{for} \; i=1, \dots, q,
\end{equation*}
where $q$ is the hidden layer dimension. As shown numerically in \Cref{app:A}, this choice of activation leads to the best performance on our problem set.

The data profile comparing the base method (\Cref{algorithm4}) and the NN-accelerated method (\Cref{algorithm5}) is given in \Cref{fig:1}. We also included in \Cref{fig:1} the variant of \Cref{algorithm5} in which Sobolev learning is dismissed, namely, problem \eqref{eq:gg1} is solved instead of \eqref{eq:gg2} in Step 1 of \Cref{algorithm3}, using the same solver and stopping criterion as for \eqref{eq:gg2}. This figure shows that our NN surrogate model outperforms the base method and that Sobolev learning provides better extrapolation steps by taking into account previously computed finite-difference gradient.

\begin{figure}[ht]
\begin{center}
\centerline{\includegraphics[scale=0.5]{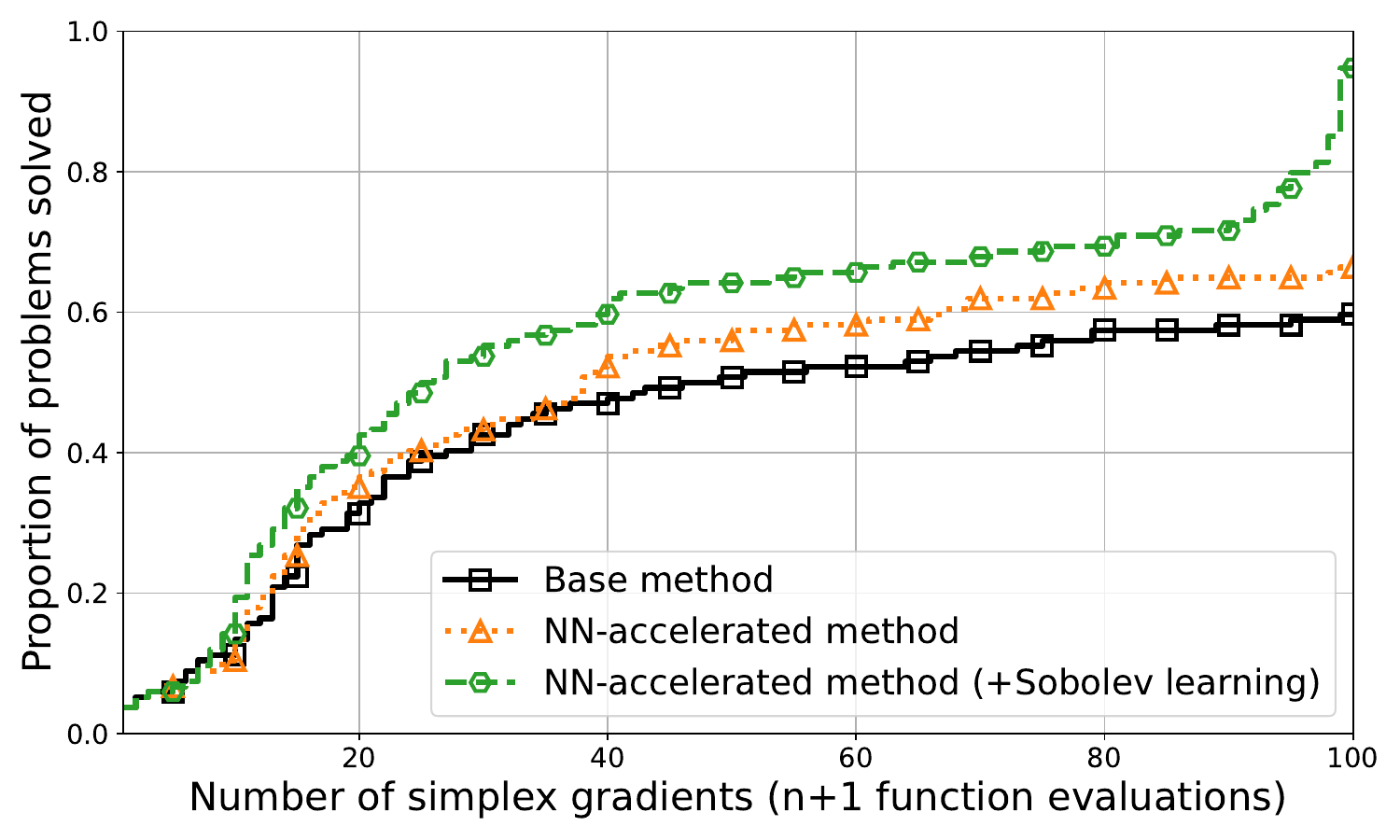}}
\caption{Data profiles comparing the base method to the surrogate-accelerated methods based on SoftPlus shallow neural networks with standard learning and Sobolev learning, for a budget of 100 simplex gradients. }
\label{fig:1}
\end{center}
\vskip -0.2in
\end{figure}

\Cref{fig:2} show the performance of the base method and the surrogate-accelerated methods with the Gaussian RBF approximation model, taking
\begin{equation*}
\psi(r) = e^{-r^2},
\end{equation*}
  using either standard learning or Sobolev learning with finite-difference gradients (other RBFs are also considered in \Cref{app:A}). Similarly to the NN-accelerated method, using a RBF surrogate outperforms the base method. Here again, Sobolev learning allows to improve further the performance of the RBF-accelerated method.

\begin{figure}[ht]
\vskip 0.2in
\begin{center}
\centerline{\includegraphics[scale=0.5]{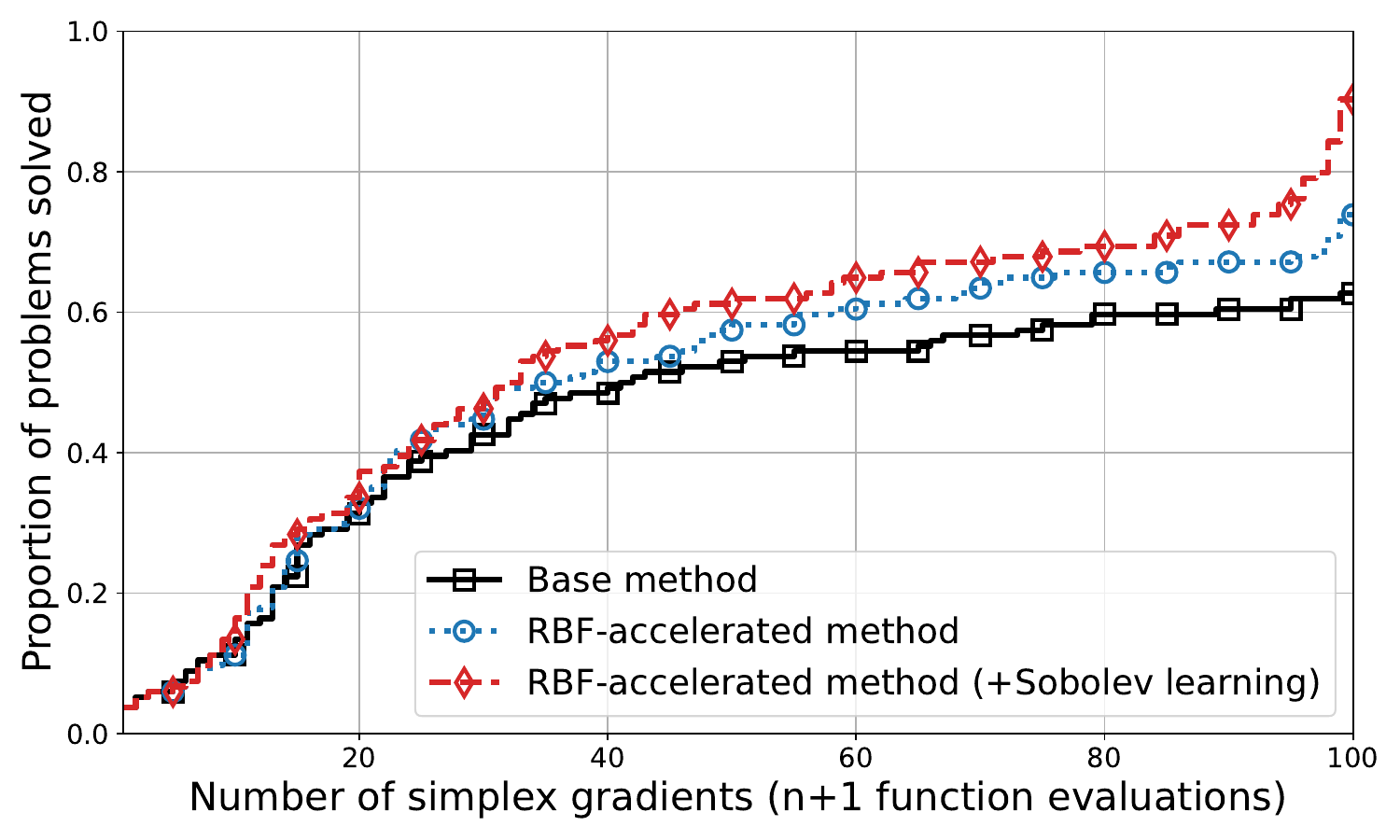}}
\caption{Data profiles comparing the base method to the surrogate-accelerated methods based on Gaussian RBF with standard learning and Sobolev learning, for a budget of 100 simplex gradients. }
\label{fig:2}
\end{center}
\vskip -0.2in
\end{figure}

To validate the worst-case complexity bounds derived in \Cref{sec:surrogate_framework}, we compute for each test problem the surrogate gain defined as
\begin{equation*}
\eta(S(T)) = \frac{1+\frac{S(T)}{2(n+1)}}{1+S(T)},
\end{equation*}
where $S(T)$ is the average number of surrogate steps over $T$ iterations. The surrogate gain $\eta(S(T)) \in [0,1]$ measures the improvement over the base method that is captured by the worst-case complexity bound. Poor surrogate models will almost never be used and therefore give a surrogate gain $\eta(S(T)) \approx 1$, hence only a small improvement in the theoretical performances of the method; in contrast, efficient surrogates will lead to $\eta(S(T)) \ll 1$ and to a significant reduction in the required number of function evaluations to reach a given accuracy. In \Cref{fig:3}, we show the distribution of $\eta(S(T_{\max}))$ on our benchmark problems set, where $T_{\max}$ is the maximum number of iterations performed by the surrogate-accelerated method on a given problem with a budget of $100$ simplex gradients.

\begin{figure}[H]
\vskip 0.2in
\begin{center}
\centerline{\includegraphics[scale=0.5]{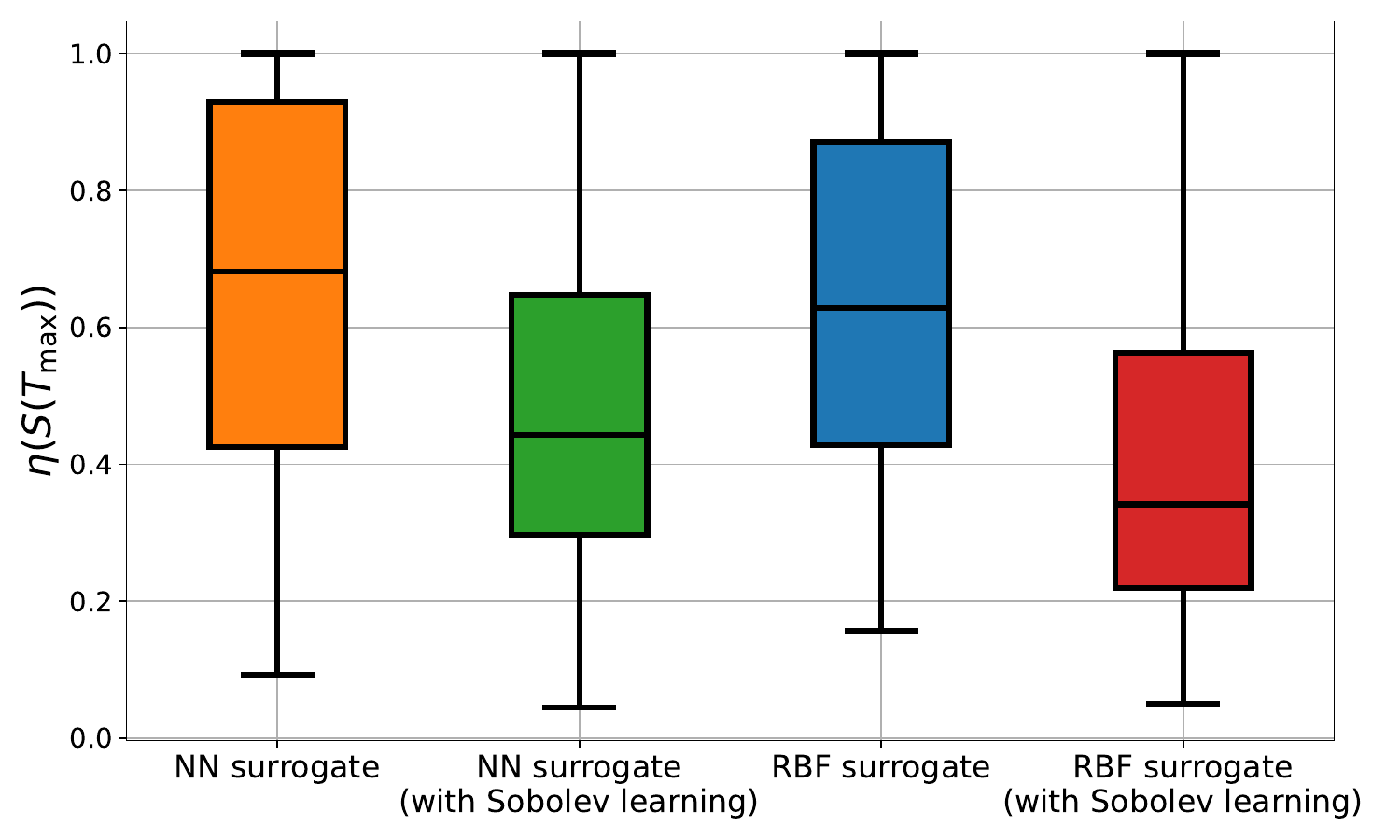}}
\caption{Box plots for the distribution of the surrogate gain over the full test set with a budget of 100 simplex gradients.}
\label{fig:3}
\end{center}
\vskip -0.2in
\end{figure}

We see that $\eta(S(T_{\max}))$ is usually smaller when we incorporate the knowledge of finite-difference gradient approximations, which agrees with the results given in \Cref{fig:1} and \Cref{fig:2}: the estimated surrogate gain $\eta(S(T_{\max}))$ has a median close to 0.4 and 0.3 respectively for the NN and RBF surrogates with Sobolev learning. Without Sobolev learning, the median surrogate gains are closer to one, with approximate values of 0.7 and 0.6 for respectively NN and RBF surrogates.

Finally, \Cref{fig:6} shows that the NN-accelerated method outperforms  the RBF-accelerated method on this test set. In both cases, the gap with respect to the base method increases overall with the maximum budget of function evaluations. Note that this contrasts with the study of the surrogate gain $\eta(S(T_{\max}))$ in \Cref{fig:3}, which predicts a slightly better complexity bound for the RBF, that is not verified numerically. This discrepancy could be explained by the fact that the RBF surrogate steps, even if accepted more often, lead to smaller reductions in the objective than the NN surrogate steps. Note also that NNs and RBFs surrogates may in practice result in very different running times for \Cref{algorithm5}, due to the difference in computational cost to train these two types of models. While for problems where the evaluations of $f$ are very costly, this computational cost will not impact much the overall running time of \Cref{algorithm5}, for applications with less costly function evaluations, the RBF-accelerated method may become substantially faster than the NN-accelerated method. 
\vspace{-0.5cm}
\begin{figure}[H]
\vskip 0.2in
\begin{center}
\centerline{\includegraphics[scale=0.5]{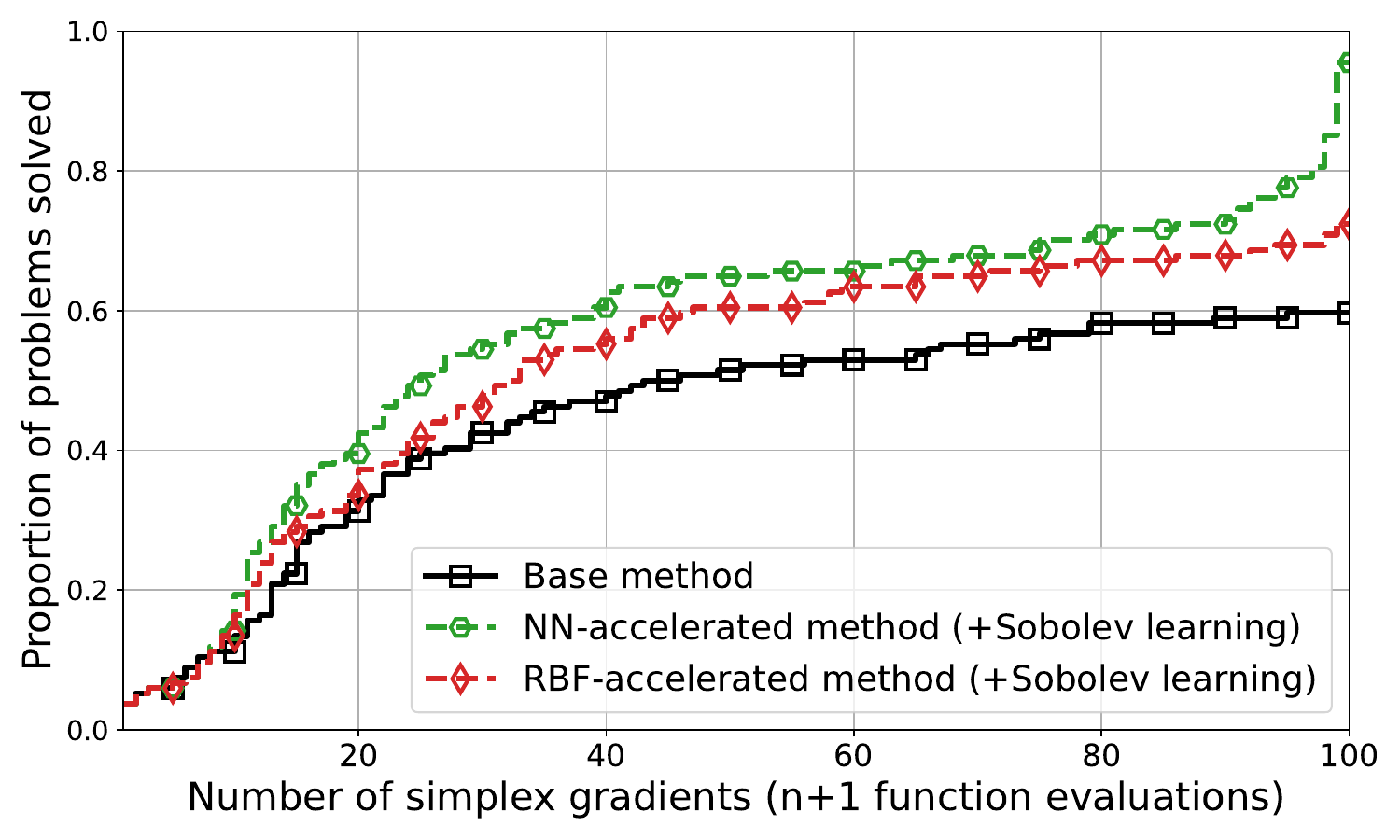}}
\caption{Data profiles comparing  the base method, the surrogate-accelerated methods based on SoftPlus shallow NN and Gaussian RBF with Sobolev learning, for a budget of 100 simplex gradients.}
\label{fig:6}
\end{center}
\end{figure}
\vspace{-1cm}
\section{Concluding remarks}
In this paper, we proposed a surrogate-based heuristic (\Cref{algorithm3}) to enhance the performance of DFO methods that rely on finite-difference gradient approximations. We incorporated the proposed heuristic into an existing DFO method (\Cref{algorithm4}) and showed that the resulting scheme (\Cref{algorithm5}) requires no more than $\mathcal{O}(n\epsilon^{-2})$ function evaluations to find an $\epsilon$-approximate stationary point of the objective function, where the constant factor in the complexity bound is inversely proportional to the average number of surrogate steps. While each outer iteration of the method costs $\mathcal{O}(n)$ function evaluations, each surrogate step costs only a single function evaluation. Thus, our theoretical results support the intuition that improved surrogates lead to greater efficiency in the method, particularly in terms of reducing calls to the zeroth-order oracle. We also tested \Cref{algorithm5} numerically with different types of models and different learning strategies. The results demonstrate that our surrogate-based heuristic can lead to a significant performance improvement over the base method, particularly when data about the approximate gradients is leveraged through Sobolev learning.

\bibliographystyle{plain}  
\bibliography{preprint}

\newpage
\appendix
\section{Supplementary numerical results}
\label{app:A}

 \Cref{fig:4} and \Cref{fig:5} compare surrogate-accelerated methods associated with different choices of activation functions and radial basis functions summarized in \Cref{table:otherModel}.

Overall the performance of the surrogate-accelerated methods does not depend much on the chosen activation or radial basis function, and surrogate-accelerated methods with Sobolev learning are robust with respect to this choice. The SoftPlus and  Gaussian RBF slightly outperform other choices, and are therefore considered in \Cref{sec:numerics} of the paper. 

\begin{figure}[H]
\begin{center}
\centerline{\includegraphics[scale=0.4]{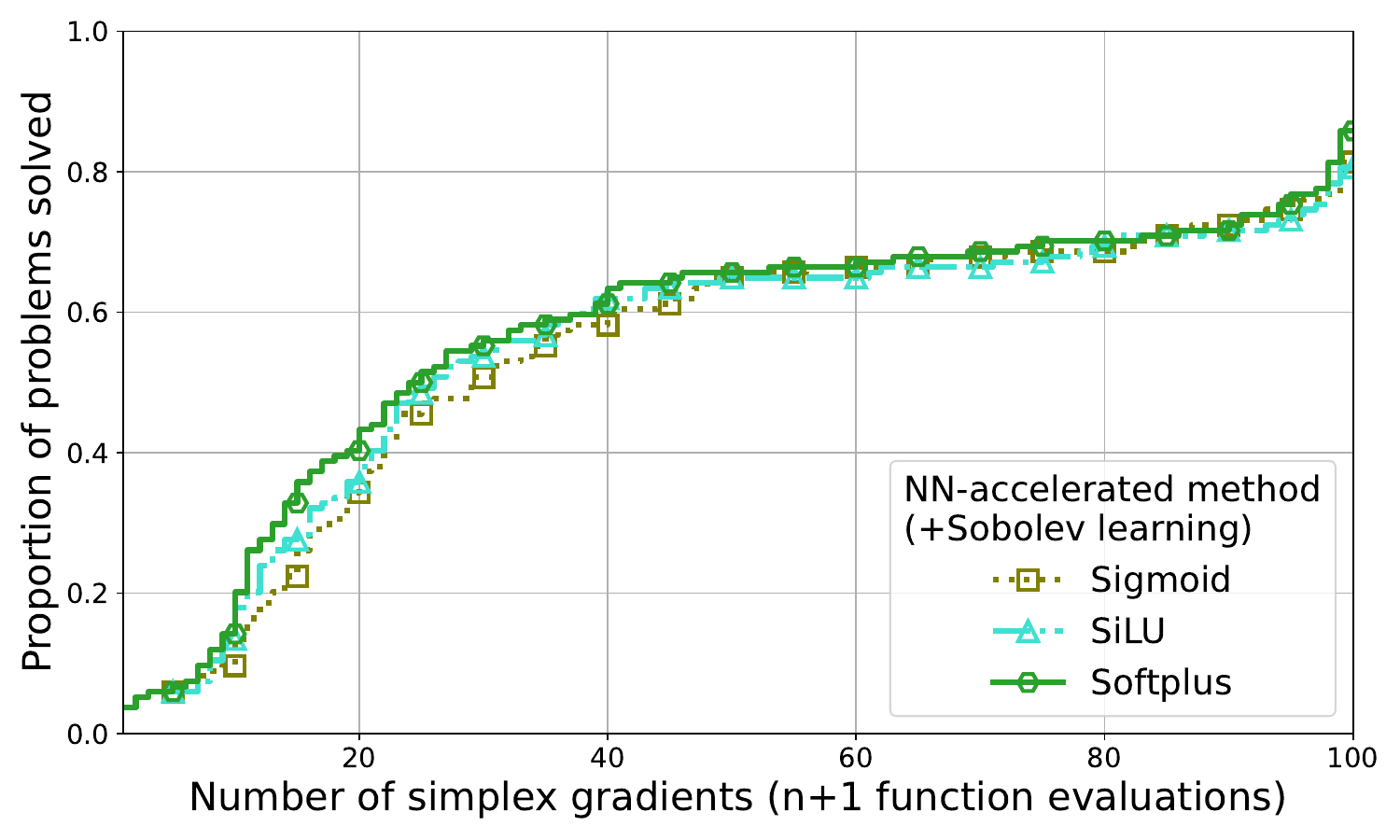}}
\caption{Data profiles for the tolerance $\tau=10^{-4}$ and budget of 100 simplex gradients. We compare different choices of NN models for the accelerated method with Sobolev learning.}
\label{fig:4}
\end{center}
\end{figure}

\begin{figure}[H]
\begin{center}
\centerline{\includegraphics[scale=0.4]{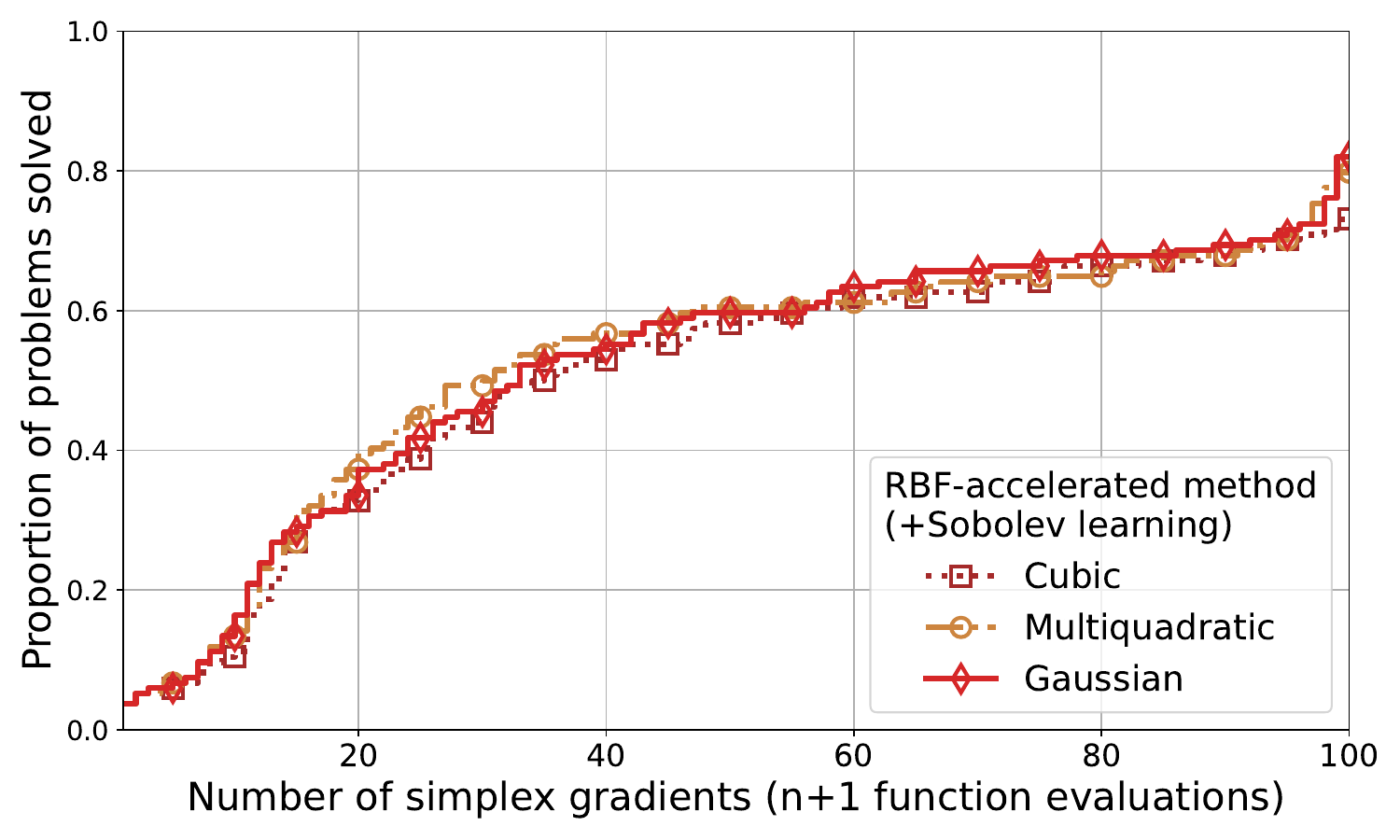}}
\caption{Data profiles for the tolerance $\tau=10^{-4}$ and budget of 100 simplex gradients. We compare different choices of RBF models for the accelerated method with Sobolev learning.}
\label{fig:5}
\end{center}
\end{figure}

\begin{table}[H]
\centering
\begin{tabular}{|c|c|}
\hline
Name & Activation / Radial basis function \\
\hline
\hline
\vphantom{\Big{|}} SoftPlus NN & $[\phi(x)]_i = \log\left( 1 + e^{[x]_i}\right)$ \\
\vphantom{\Big{|}} Sigmoid NN & $[\phi(x)]_i = \frac{1}{1 + e^{[x]_i}}$ \\
\vphantom{\Big{|}} SiLU NN & $[\phi(x)]_i =  \frac{[x]_i}{1 + e^{[x]_i}}$ \\
\hline
\vphantom{\Big{|}} Gaussian RBF & $\psi(r) = \exp( -r^2)$ \\
\vphantom{\Big{|}} Multiquadratic RBF & $\psi(r) = -\sqrt{1+r^2}$ \\
\vphantom{\Big{|}} Cubic RBF & $\psi(r) = r^3$ \\
\hline
\end{tabular}
\caption{List of NN and RBF surrogate models considered.}
\label{table:otherModel}
\end{table}



\end{document}